\documentclass[11pt]{amsart}
\newif\ifextended

\usepackage{upgreek}
\usepackage{hyperref} 
\usepackage{amsthm}
\usepackage[margin=1.2in]{geometry}

\let\oldtocsubsection=\tocsubsection
\renewcommand{\tocsubsection}[2]{\hspace{1.8em}\oldtocsubsection{#1}{#2}}

\usepackage{tikz, tikz-cd}
\usetikzlibrary{arrows, decorations.pathreplacing,calc}

\tikzset{%
	add/.style args={#1 and #2}{to path={%
		 ($(\tikztostart)!-#1!(\tikztotarget)$)--($(\tikztotarget)!-#2!(\tikztostart)$)%
	 \tikztonodes}}
 }

\usepackage{amssymb,amsmath}
\usepackage{mathtools}
\usepackage{enumerate}
\usepackage{mathrsfs} 
\usepackage{cancel}
\usepackage[normalem]{ulem}
\usepackage{wasysym}
\usepackage{dsfont}

\renewcommand{\emptyset}{\varnothing}
\newcommand{\tri}{\triangle}
\newcommand{\bd}{\partial} 
\newcommand{\cl}[1]{\overline{#1}} 
\newcommand{\then}{\Longrightarrow} 

\DeclareMathOperator{\Isom}{Isom} 
\DeclareMathOperator{\Stab}{Stab} 
\DeclareMathOperator{\Aut}{Aut}

\DeclareMathOperator{\PSL}{PSL}
\DeclareMathOperator{\SL}{SL}

\DeclareMathOperator{\PGL}{PGL}

\DeclareMathOperator{\diag}{diag}

\DeclareMathOperator{\inj}{inj}
\DeclareMathOperator{\vol}{vol}

\newcommand{\field}[1]{\mathbb{#1}} 

\newcommand{\R}{\field{R}} 
\renewcommand{\H}{\field{H}} 
\newcommand{\Z}{\field{Z}} 
\newcommand{\N}{\field{N}} 
\newcommand{\T}{\field{T}} 
\newcommand{\RP}{\R\mathbb P} 

\newcommand{\Gam}{\Gamma}
\newcommand{\gam}{\gamma}
\newcommand{\Om}{\Omega}

\newcommand{\Lam}{\Lambda}

\newcommand{\ep}{\epsilon}
\newcommand{\sig}{\sigma} 
\newcommand{\Sig}{\Sigma} 
\renewcommand{\phi}{\varphi}
\renewcommand{\tau}{\uptau}

\numberwithin{equation}{section}
\numberwithin{figure}{section}
\numberwithin{table}{section}

\theoremstyle{plain}
\newtheorem{thm}{Theorem}
\numberwithin{thm}{section}
\newtheorem{theorem}[thm]{Theorem}
\newtheorem{prop}[thm]{Proposition}
\newtheorem{proposition}[thm]{Proposition}

\newtheorem{corollary}[thm]{Corollary}
\newtheorem{lem}[thm]{Lemma}
\newtheorem{lemma}[thm]{Lemma}

\newtheorem*{theorem*}{Theorem}
\newtheorem*{proposition*}{Proposition}

\newtheorem{fact}[thm]{Fact}

\newcounter{intro}

\theoremstyle{definition}

\newtheorem{definition}[thm]{Definition}

\newcommand{\blue}{blue!60!green}

\hypersetup{colorlinks=true, linkcolor=\blue, urlcolor=\blue, citecolor=\blue}

\title{Geodesic flow of nonstrictly convex Hilbert geometries}
\author{Harrison Bray} 

\begin{document}

 \begin{abstract}
 	In this paper we describe the topological behavior of the geodesic flow for a class of 
	closed 3-manifolds realized as quotients of nonstrictly convex
	Hilbert geometries, constructed and described explicitly by Benoist. 
	These manifolds are Finsler geometries which have 
	isometrically embedded flats, but also some hyperbolicity and an explicit geometric structure. 
	We prove the geodesic flow of the
	quotient is topologically mixing and satisfies a nonuniform Anosov closing lemma, 
	with applications to entropy and orbit counting. We also prove entropy-expansiveness for the
	geodesic flow of any compact quotient of a Hilbert geometry.
 \end{abstract}

 \maketitle

 \section{Introduction}
 \label{sec:intro}
 
 We study topological behavior of the geodesic flow of a class of closed 3-manifolds
 which are only Finsler,
 meaning the tangent space admits a norm which does not come from an inner product, and for which
 the geodesic flow is nonuniformly hyperbolic due to the presence of isometrically embedded flats
 of dimension two. The 3-manifolds arise as quotients of 
 properly convex domains in real projective space by discrete groups of projective transformations. 
 Such objects are known as {\em Hilbert geometries} or {\em convex real projective
 structures}. The structure of the domain and the quotient is well-described thanks to Benoist
 (\cite{Ben4}, see Theorem \ref{thm:ben3mfldgeom}). As such, we refer to the 3-manifolds of
 interest as \emph{Benoist 3-manifolds}. 

 We prove several recurrence properties of the geodesic flow of the Benoist 3-manifolds reminiscent of
 hyperbolic dynamics, such as topological transitivity and a nonuniform Anosov Closing Lemma. Though
 stable and unstable sets are not even defined for a dense set of points, we prove that strong unstable
 leaves are dense for closed \emph{hyperbolic} orbits, which are dense in the phase space. These
 results culminate in the following:

 \begin{theorem*}[Theorem \ref{thm:topologicalmixing}]
	 The geodesic flow of a Benoist 3-manifold is topologically mixing\footnote{
	 A continuous dynamical system $f^t\colon X\to X$ is
	 \emph{topologically mixing} if for any open $U,V\subset X$ there exists a $T>0$ such that
 	$f^T(U)\cap V\neq\emptyset$. }.
 \end{theorem*}

 The geometric properties of the universal cover which Benoist verifies in dimension three are
 essential for the arguments, hence the results do not immediately generalize. 

 This paper also serves as a precursor to work of the author on the Bowen-Margulis measure of maximal
 entropy \cite{braymme}. To that end, we verify conditions of Bowen \cite{bowen72} for easier computability of
 topological entropy which holds generally: 
 \begin{theorem*}[Theorem \ref{thm:hexp}]
	 The geodesic flow of any closed Hilbert geometry satisfies Bowen's entropy-expansive
	 property. 
 \end{theorem*}

 Then we have the following consequence for the Benoist 3-manifolds, a corollary of which is
 positive topological entropy. 

 \begin{proposition*}[Proposition \ref{prop:entropyineq}]
	 The topological entropy of the geodesic flow of a Benoist 3-manifold is bounded below by
	 the exponential growth rate of lengths of hyperbolic closed orbits. 
 \end{proposition*}

 The structure of the paper is as follows: we first introduce the objects of interest and the
 relevant background in Section \ref{sec:background}. In Section \ref{sec:automorphisms}
 we study automorphisms of the universal cover, and prove that the additive subgroup of $\R$
 generated by lengths of closed hyperbolic orbits is dense 
 (Proposition \ref{prop:densespec}). This result, along with transitivity (Propositive
 \ref{prop:toptrans}) and nonuniform Anosov
 Closing (Theorem \ref{thm:acl}) from Section \ref{sec:recurrence} will be crucial
 for the proof of topological mixing in Section \ref{sec:mixing}. In the same section we also prove a
 nonuniform orbit
 gluing lemma (Lemma \ref{lem:orbitglue}) which suffices for topological mixing but requires no
 control over exponential contraction or expansion. Section \ref{sec:hexp} is devoted to the proof
 of entropy-expansiveness and Section \ref{sec:counting} to orbit counting, with  
 remarks on the relationships between the topological entropy, the volume entropy, and the
 critical exponent of the group.

 \subsection*{Acknowledgements}
 The author is very grateful to Micka\"el Crampon, who contributed significantly to Theorems
 \ref{thm:topologicalmixing} and \ref{thm:hexp} in particular, and to Ludovic Marquis for his
 guidance with Lemma
 \ref{lem:baker_cooper_irred}. The author also thanks his thesis advisor Boris Hasselblatt and postdoc
 mentor Dick Canary for their suggestions and mentoring, and Ralf Spatzier for helpful discussions.
 In addition, the author is grateful to
 the CIRM for support in the early stages of this work. The author 
 was supported in part 
 by NSF RTG grant 1045119.

 \section{Background}
 \label{sec:background}
 	A domain $\Om\subset\RP^n$ is \emph{proper} if there is an affine chart in
	which $\Om$ is bounded, and {\em properly convex} if moreover $\Om$ is convex in this affine
	chart, meaning the intersection of $\Om$ with any line is connected. Similarly, $\Om$ is
	{\em strictly convex} if the intersection of $\bd\Om$ with any line in the compliment
	of $\Om$ contains at most one point. 
	A \emph{Hilbert geometry} on a properly convex domain $\Om\subset \RP^n$ is determined by the
	Hilbert metric,
	defined on an affine chart for $\Om$ as follows:  
	for any $x, y\in \Om$, there is a unique projective line $\cl{xy}$  passing through $x$ and
	$y$. Take $a$ and $b$ to be the intersection points with $\bd\Om$. 
	Then the \emph{$\Om$-Hilbert distance} between $x$ and $y$ is 
 \[
	 d_\Om(x,y) := \frac12 \left| \log [a,x,y,b] \right|,
 \]
 where $[a,x,y,b]:= \frac{|ay|}{|ax|} \frac{|bx|}{|by|}$.
	One can verify that $d_\Om$ satisfies the properties of a metric, is complete on $\Om$, and is well-defined for any affine representation of $\Om$
	by projective invariance of the cross-ratio. 
	Projective lines are always geodesic in this metric, but not all geodesics are lines. 

	The Hilbert metric is compatible with a Finsler norm, which is Riemannian only when $\Om$ is
	an ellipsoid. 
	One can compute that for $(x,v)\in T\Om$, the Finsler norm is given by
	\[
		F(x,v) := \frac1{|v|}\left( \frac1{|xv^+|}+\frac1{|xv^-|}\right)
	\]
	where $v^-,v^+$ are the intersection points with the topological boundary $\bd\Om$ of the
	projective line determined by $v$. 
	A properly convex domain $\Om$ in $\RP^2$ is uniquely geodesic if and only
	if there is at most one open line segment in $\bd\Om$ (this can be verified using the
	well-definedness of the cross-ratio of four lines, or see also \cite{braymme}). 
	The ellipsoid in $\RP^n$ is isometric to $\H^n$ when endowed with the
	Hilbert metric. In this metric, angles are defined, though distorted, since
	the Finsler norm is Riemannian. This model for hyperbolic space is known as
	the Beltrami-Klein model.

For a properly convex open $\Om \subset \RP^n$, define the \emph{automorphism group} of $\Om$ to be
 \[
 \Aut(\Om) := \{ g\in \PSL(n+1, \R) \mid g\Om = \Om\} .
 \]

 Note that $\Aut(\Om)<\Isom(\Om)$, the isometry group of $(\Om,d_\Om)$, since projecive
 transformations preserve the cross-ratio. The full isometry group of $(\Om,d_\Om)$ is, up to index
 2, the group of collineations which preserve $\Om$ \cite{walsh}.
	A properly convex domain $\Om\subset\RP^n$ is \emph{divisible} if it admits a cocompact
	action by a discrete subgroup $\Gam$ of $\PSL(n+1,\R)$, in which case we say $\Gam$ \emph{divides} $\Om$. 

	As a first example, the ellipse is divisible by any Fuchsian group. The projective triangle,
	isometric to $\R^2$ with a hexagonal norm when endowed with the Hilbert metric
	\cite{delaHarpe},
	admits a $\Z^2$ action with quotient torus. 

Suppose $\Gam < \PSL(n+1,\R)$ acts properly discontinuously without torsion on $\Om\subset \RP^n$, so that the quotient $M=\Om/\Gam$ is a manifold. The \emph{geodesic flow} of $M$ is defined on $SM$, the Finsler unit tangent bundle to $M$, by flowing unit tangent vectors along projective lines at unit Hilbert speed:
\begin{align*}
 \phi^t\colon SM&\longrightarrow SM \\
(x, v)&\longmapsto (x+tv, v).
\end{align*}
 
In other words, $(x,v)\in SM$ determines a unique oriented projective line $\ell_v\colon\R\to M$ parameterized
at unit Hilbert speed, with $\ell_v(0)=x$ and $\phi^t(v)$ the Finsler unit tangent vector to
$\ell_v$ based at $\ell_v(t)$.
In the strictly convex case, geodesics are uniquely projective lines and this definition coincides
with the standard definition of geodesic flow. In our setting, geodesics are not unique so we
require defining the geodesic flow in this very natural way. Note also from the definitions that the
regularity of the boundary of $\Omega$ determines the regularity of the geodesic flow. 
\subsection{Benoist's dichotomy}

The following  landmark theorem of Benoist for the study of divisible Hilbert geometries is equivalence of the regularity of the boundary, convexity of the boundary, and hyperbolicity of the flow based on abstract properties of the group.

\begin{theorem}
	[{\cite[Theorem 1.1]{Ben1}}]
	\label{thm:benoistsdichotomy}
	Suppose $\Gam$ is a discrete torsion-free subgroup of $\PSL(n+1,\R)$ dividing an open properly convex domain $\Om\subset \RP^n$. Then the following are equivalent:
	\begin{enumerate}[(i)]
		\item The domain $\Om$ is strictly convex.
		\item The boundary $\bd\Om$ is of class $C^1$.
		\item The group $\Gam$ is $\delta$-hyperbolic.
		\item The geodesic flow on the quotient manifold $M=\Om/\Gam$ is Anosov.
	\end{enumerate}
\end{theorem}

Essential to Benoist's theorem is Benzecri's thesis work on the $\PGL$-orbits of marked properly convex sets in projective space \cite{Benz60}. 
In fact, an application of the work of Benzecri shows that in dimension two, a divisible $\Om$ is either strictly convex with $C^1$-boundary or a projective triangle.

\subsection{The Benoist 3-manifolds}

A natural question is whether, 
as in dimension two, a divisible Hilbert geometry in any dimension is
either strictly convex with $C^1$-boundary or a simplex. 
Benoist proved, in the contrary, 
existence of 
Hilbert geometries in dimension three which are nonstrictly convex and indecomposable via a modification of the Kac--Vinberg Coxeter construction \cite[Proposition 1.3]{Ben4}.
Moreover, Benoist proved geometric properties of such Hilbert geometries.
Before stating the theorem, we define some terms: 
let $C$ be the cone in $\R^{n+1}$ over a properly convex domain $\Omega$ in $\RP^n$, and define
$C$ to be properly convex if and only is $\Omega$ is. Then $\Omega$ is \emph{decomposable} if there
exist vector subspaces $V_1,V_2 \subset \R^{n+1}$ and properly convex cones $C_1\subset V_1,
C_2\subset V_2$ such that $C=C_1+C_2$. Else, $\Omega$ is \emph{indecomposable}. Note that a simplex
is always decomposable.

A \emph{properly embedded triangle} in $\Om$ 
is a projective triangle $\tri \subset \Om$ such that $\bd\tri\subset \bd\Om$. Let $\mathcal T$
denote the collection of triangles $\tri$ properly embedded in $\Om$, and $\Gam_\tri :=
\Stab_\Gam(\tri) = \{\gam\in\Gam \mid \gam \tri=\tri\}$ be the subgroup of $\Gam$ stabilizing
$\tri\in \mathcal T$. 

\begin{theorem}[{\cite[Theorem 1.1]{Ben4}}]
	\label{thm:ben3mfldgeom}
	Let $\Gam < \SL(4,\R)$ be a discrete torsion-free subgroup which divides an open, properly
	convex, 
	indecomposable 
	$\Om\subset \RP^3$, and $M = \Om/\Gam$. Then 
	\begin{enumerate}[(a)]
		\item Every subgroup in $\Gam$ isomorphic to $\Z^2$ stabilizes a unique triangle
			$\tri \in \mathcal T$. 
		\item If $\tri_1, \tri_2\in\mathcal T$ are distinct, then
			$\cl{\tri_1}\cap\cl{\tri_2} = \varnothing$. 
		\item For every $\tri\in \mathcal T$, the group $\Gam_\tri$ contains an index-two
			$\Z^2$ subgroup. 
		\item The group $\Gam$ has only finitely many orbits in $\mathcal T$. 
		\item The image in $M$ of triangles in $\mathcal T$ is a finite collection
			$\mathcal F$
			of disjoint tori and Klein bottles. 
			If one cuts $M$ along
			each $\T\in \mathcal F$, each of the resulting connected components is
			atoroidal.
		\item Every nontrivial line segment is included in the boundary of some
			$\tri\in\mathcal T$. 
		\item If $\Om$ is not strictly convex, then the set of vertices of triangles in
			$\mathcal T$ is dense in $\bd\Om$. 
	\end{enumerate}
\end{theorem}

We will call the compact quotients of nonstrictly convex, indecomposable, divisible Hilbert
geometries in dimension three \emph{Benoist 3-manifolds}. 
The topological decomposition as in \ref{thm:ben3mfldgeom}(e) is an example of a Jaco--Shalen--Johansson
(JSJ) decomposition \cite{jaco-shalen,johannson}. Benoist remarks after Theorem
\ref{thm:ben3mfldgeom} that as a
consequence of Thurston's geometrization, the
atoroidal components of the quotient are diffeomorphic to finite volume hyperbolic 3-manifolds
\cite[pp. 4-5]{Ben4}.

  \section{Properties of automorphisms of Benoist's 3-manifolds}
  \label{sec:automorphisms}

Let $\Om\subset\RP^n$ be a properly convex domain. Then for any $g\in\Aut(\Om)$, we can
define the \emph{translation length} of $g$ by 
\[
	\tau(g):= \inf_{x\in\Om} d_\Om(x,g.x).
\]
An {\em axis} of $g$ is a $g$-invariant projective line in $\Omega$.

We will diverge slightly from the literature here in our terminology. 
We define 
$g\in\Aut(\Om)$ to be \emph{hyperbolic} if $\tau(g)>0$ and the infimum is attained
along a unique axis of $g$.  
Any other $f\in\Aut(\Om)$ for which
$\tau(f)$ is positive and realized, but not along an axis of $f$, 
will be
called \emph{flat}. Typically, both flat and hyperbolic automorphisms are called hyperbolic 
(c.f. \cite{CrMar14}), and in 
the strictly convex case there would be no need for the distinction we introduce here. 
A projective transformation is \emph{quasi-hyperbolic} if $\tau(g)>0$ and the infimum is not
attained, \emph{parabolic} if $\tau(g)=0$ and the infimum is not attained, and \emph{elliptic} if
$\tau(g)=0$ and the infimum is attained.


There is an important property of the Benoist 3-manifolds which has dynamical implications for the
group elements. If $\Omega/\Gamma$ is a Benoist 3-manifold, then $\Omega$ must be indecomposable,
hence $\Gamma$ is irreducible \cite{vey70}.  
	A subgroup $H<\PSL(4,\R)$ is \emph{irreducible} if it does not stabilize a projective point,
	line, or plane in $\RP^3$, and $H$ is \emph{strongly irreducible} if every finite-index
	subgroup is irreducible.  
Since $\Gam$ is irreducible, all
elements of
$\Gam$ are {\em positively proximal} \cite[Proposition 1.1]{BenCC}, meaning their top eigenvalues are real
and positive
and have  one-dimensional eigenspaces. Since $\Gam$ is a group, each $g\in\Gam$ must in fact be {\em
biproximal},
meaning the top and bottom eigenvalues each have one-dimensional eigenspaces, and we note that they
must also both be real and positive. 
In fact, for the Benoist 3-manifolds, all elements of any $\mathbb Z^2$ subgroup
of $\Gam$ have only real positive eigenvalues \cite[Corollary 2.4]{Ben4}.
\begin{proposition}
	Let $M=\Om/\Gam$ be a Benoist 3-manifold with discrete, torsion-free dividing group $\Gam$.
	Then 
	there are no quasi-hyperbolic group elements in $\Gam$. 
	\label{prop:noquasihyp}
\end{proposition}
\begin{proof}

	Suppose $\tau(g)>0$. Let the eigenvalues of a representative of $g$ in $\SL(4,\R)$ be given by 
	$\lambda_i(g)$ such that 
	$\lambda_0(g)>|\lambda_1(g)|\geq |\lambda_2(g)|
	> \lambda_3(g)$. 
	It is straightforward to
	verify that since $g$ is biproximal, 
	$\tau(g)=\log \frac{\lambda_0(g)}{\lambda_3(g)}$ and is realized
	along a projective line joining the eigenvectors associated to 
	$\lambda_0$ and $\lambda_3$. 
	Again since $g\in\Aut(\Om)$, we may choose a maximal projective line segment in $\cl{\Om}$ preserved by
	$g$ along which $\tau(g)$ is realized, which we denote  
	$\ell_g$. 
	To be quasi-hyperbolic, $\ell_g$ must be contained in $\bd\Om$.
	Then by Theorem \ref{thm:ben3mfldgeom}(f), $\ell_g\subset\bd\tri$ for some properly embedded
	triangle $\tri$
	in $\Om$. 
	  Since $g$ acts by projective transformations and preserves $\Omega$, for any properly
	  embedded triangle $\tri$ we have that $g\tri$ is also a properly embedded triangle. Then since
	  $g$ stabilizes $\ell_g$ we have $\ell_g\subset g\cl{\tri}\cap\cl{\tri}\neq\varnothing$
	  implying
	  $g\tri=\tri$ by Benoist's Theorem \ref{thm:ben3mfldgeom}(b). 
	Then $g\in\Stab(\tri)$ and thus $\tau(g)$ is realized in $\tri$.
\end{proof}

We now introduce new terminology for points in $\bd\Omega$.  
We say that $\xi\in\bd\Om$ is \emph{proper} if there is a unique supporting hyperplane to $\Om$ at
$\xi$, where a hyperplane $H\subset \RP^n\smallsetminus \Om$ is a {\em supporting hyperplane} at $\xi$ if $\xi\in
H\cap \bd\Om$. Also, $\xi\in\bd\Om$ is \emph{extremal} if there is no open line segment containing $\xi$ embedded
in $\bd\Om$; in other words, $\xi$ is extremal if for any supporting hyperplane $H$ at $\xi$, we have 
$H\cap
\bd\Om=\{\xi\}$.
Note that by Benoist's Theorem \ref{thm:ben3mfldgeom}(f) and duality, proper extremal
points form the compliment of the boundaries of properly embedded triangles. 

\begin{prop}
	Let $M=\Om/\Gam$ be a Benoist 3-manifold with discrete, torsion-free dividing group $\Gam$.
	Then for all $g\in\Gam$, 
        \begin{itemize}
		\item $g$ is {hyperbolic} if and only if $g$ has exactly two fixed points $g^-,g^+$
		  in $\cl{\Om}$ which are proper extremal points in the boundary. These fixed points
		  are respectively repelling and attracting under the dynamics of $g$ on
		  $\cl{\Om}$. 
	\item $g$ is flat if and only if $g\in\Stab(\tri)$ for some properly embedded $\tri$. 
        \end{itemize}
These are the only possible automorphisms of a divisible, 
indecomposable 
domain in $\RP^3$.        
        \label{prop:nscisoms}
\end{prop}
\begin{proof}
	Since $\Gam$ is discrete and torsion-free, there are no elliptic isometries in $\Gam$. Since $M$ is
	compact, there are no parabolic isometries in $\Gam$ (see also \cite[Lemma 2.8]{Ben4}).
	By Proposition \ref{prop:noquasihyp}, there are no quasi-hyperbolic elements of $\Gam$. 
	Thus, it suffices to characterize the dynamics of group elements with translation length realized in
	$\Om$. 
	Since all elements of $\Gamma$ are biproximal, this is straightforward.
\end{proof}

\subsection{Lengths of hyperbolic orbits}
The goal of this subsection is to prove that the additive subgroup of $\R$ generated by lengths of
closed hyperbolic orbits is dense via 
Zariski density of an immersed hyperbolic surface group in $\Gam$. 

\begin{fact}[{\cite{masterszhang,BaCo15}}]
The fundamental group of a complete, finite volume, noncompact 
hyperbolic 3-manifold contains a closed hyperbolic quasi-Fuchsian 
surface subgroup. 
\label{fct:surfacesubgroup}
\end{fact}
Let $\Gam_{\text{hyp}}$ denote the hyperbolic elements of $\Gam$, and let $\Sigma<\PSL(4,\R)$ be the
subgroup of $\Gam$ which is isomorphic to the hyperbolic surface subgroup given by 
Fact
\ref{fct:surfacesubgroup} and Benoist's remark following Theorem \ref{thm:ben3mfldgeom}. 
Since $\Sigma$ is a quasi-Fuchsian subgroup, 
no element of $\Sigma$
can preserve 
any properly embedded triangle. Then by Proposition \ref{prop:nscisoms}, 
$\Sigma$ is a subgroup in $\Gam_{\text{hyp}}$. 
\begin{corollary}
	There exist infinitely many noncommuting hyperbolic group elements in $\Gam$. 
	\label{cor:exist_hyp_elts}
\end{corollary}

Let $G$ be any subset of $\Aut(\Om)$ 
and $\mathcal L(G):=\langle \tau(g) \rangle_{g\in G}$ the additive subgroup of $\R$ generated by
translation lengths of group elements in $G$ acting on $\Om$. Note that if $G$ is a subset of the 
group $\Gam$ which divides $\Om$ then $\mathcal L(G)$ is the additive subgroup of $\R$ generated by lengths
of closed curves in $\Om/\Gam$ associated to conjugacy classes in $\Gam$ of elements of $G$.

\begin{corollary}[{of \cite[Fact 5.5]{Ben1}}]
	If $\Gam$ is a Zariski dense subgroup of $\SL(n+1,\R)$ preserving a properly
	convex domain $\Om\subset\RP^n$, then $\mathcal L(\Gam)$ is dense in $\R$. 
	\label{cor:ben_spec_dense}
\end{corollary}

	If $\Om$ is not an ellipsoid, then the hypotheses of Corollary \ref{cor:ben_spec_dense} 
	hold whenever $\Gam$ is acting cocompactly on an {indecomposable} properly convex and
	{\em strictly convex} $\Om$ in
	projective space {\cite[Theorem 1.2]{Ben2}}. 
	In our case, $\Om$, 
	the universal cover of a Benoist 3-manifold is
	indecomposable but is not strictly convex, so Corollary \ref{cor:ben_spec_dense} does not 
	apply directly to $\Gam$ the fundamental group of a Benoist 3-manifold. 

	\begin{proposition}[restatement of {\cite[Proposition 6.5]{CrMar14}}]
	Suppose $\Gam$ is a strongly irreducible subgroup of $\SL(n+1,\R)$ which preserves a
	properly convex $\Om\subset \RP^n$. Let $G$ be the Zariski closure of $\Gam$. Then $G$ is a
	Zariski-connected real semi-simple Lie group. 
		\label{prop:restatementCrMar}
	\end{proposition}

	If $\log(G)=\{ \log|\lambda_1(g)|, \log|\lambda_2(g)|,\ldots,\log|\lambda_{n+1}(g)| \mid
	g\in G\}$ with $\lambda_i$ decreasing in magnitude, 
	Then Proposition \ref{prop:restatementCrMar} implies $\log(G)$ is a subspace of
	$\R^{n+1}$, and $\log(\Gamma)$ is dense in $\log(G)$. It is straightforward to verify that
	$\mathcal L(\Gam)\neq \R$ implies $\overline{\log(\Gam)}$ cannot be a subspace of
	$\R^{n+1}$.
	Thus, it suffices to prove that $\Sigma$ is strongly irreducible to conclude that $\mathcal
	L (\Sigma)$ is dense in $\R$. 

	For the following lemma, we will use that $\Gam$ preserves
	(divides) $\Om$ if and
	only if the transpose $\Gam^t$ preserves (divides) the projective dual $\Omega^\ast$
	\cite[Lemma 2.8]{Ben1}.

\begin{lemma}
	The closed hyperbolic surface subgroup $\Sigma$ is either strongly irreducible or 
	$\mathcal L(\Sigma)$ is dense in $\R$. 
	\label{lem:baker_cooper_irred}
\end{lemma}

\begin{proof}
	First, since $\Sigma$ is a surface group, every finite-index subgroup is also a surface subgroup. 
	It suffices to show any surface group in $\PSL(4,\R)$ preserving a domain $\Om\subset \RP^3$ is irreducible. 
	By contradiction, suppose $\Sigma$ fixes a point $p\in\RP^3$. 
	Clearly $p\not\in\Om$ because $\Gam$ acting discretely without torsion cannot have elliptic elements. 
	Also, $p\not\in\bd\Om$ because elements of $\Sigma$ 
	do not stabilize any triangles so all fixed points of elements of $\Sigma$ are proper and
	extremal, and noncommuting hyperbolic isometries cannot fix the same proper extremal point
	since $\Gam$ 
	acts properly discontinuously on $\Om$.
	If $p\not\in\cl{\Om}$, then we consider the dual case: when $\Sig^t$ preserves a projective
	plane $\Pi$ which intersects $\Om^\ast$. 
	Then $\Sigma^t$ is acting cocompactly on a totally geodesic hypersurface $\Pi\cap \Om^\ast$. 
	By {\cite[Theorem 1.2]{Ben2}},  
	$\Sigma^t$ is either Zariski dense and hence $\mathcal L(\Sigma^t)$ is dense 
	by Corollary \ref{cor:ben_spec_dense} or $\Pi\cap \Om^\ast$ is
	homogeneous and $\mathcal L(\Sigma^t)$ is dense in $\R$ anyways. 
	Then so is $\mathcal L(\Sigma)$ 
	since dual groups preserving dual properly convex sets are isospectral. 
	Thus, if $\Sigma$ preserves $\Om$ and fixes a point, then $\mathcal L(\Sigma)$ is dense in
	$\R$. 

	Now suppose $\Sigma$ preserves a line $l$. 
	The case where $l\subset \Om$ or $l$ is disjoint from $\cl{\Om}$ by duality is impossible because $\Aut(l)=\R$. 
	If $l$ intersects $\bd\Om$ then either $\Sigma \not< \Aut(\Om)$ or $\Sigma \subset
	\Stab(\tri)$, both a contradiction. 

	If $\Sigma$ stabilizes a plane, then we revisit the dual cases where $\Sigma^t$ stabilizes a
	point, unless the plane intersects $\Om$. In this case, we have already seen that $\mathcal
	L(\Sigma)$ is dense in $\R$.  
\end{proof}

	\begin{proposition}
		\label{prop:densespec}
		Let $\Om/\Gam=M$ be a Benoist 3-manifold.
		Then $\mathcal L(\Gam_{\text{hyp}})$ is dense in $\R$. 
	\end{proposition}
	\begin{proof}
		By Lemma \ref{lem:baker_cooper_irred} and Proposition \ref{prop:restatementCrMar}, 
		the group $\Sigma\subset\Gam_{\text{hyp}}$ is either Zariski dense or $\mathcal
		L(\Sigma)$ is dense in $\R$. 
		By Corollary \ref{cor:ben_spec_dense}, density of $\mathcal L(\Sigma)$ holds in both cases. 
	\end{proof}

\section{Recurrence behavior}
\label{sec:recurrence}

Recall that a point $\xi\in\bd\Om$ is {proper} if there exists a unique supporting hyperplane to $\Om$ at $\xi$ 
and {extremal} if $\xi$ is contained in no open line segment embedded in $\bd\Om$. 
For the Benoist 3-manifolds, vertices of properly embedded triangles are the only nonproper points, and 
all nonextremal points are contained in the side of some properly embedded triangle.
Thus, the proper extremal points are the complement of the boundaries of properly embedded triangles. 

We will say $v\in S\Om$ is \emph{forward regular} if $v^+$ is a proper extremal point, and similarly for \emph{backward regular}. If $v^+$ is a nonproper or nonextremal point, then $v$ is \emph{forward singular}.  
If $v$ is both forward and backward regular, then we will say $v$ is \emph{regular} (and similarly
for singular vectors). 	
Let $S\Om_{\text{reg}}$ be the collection of regular vectors and the complement, $S\Om_{\text{sing}}$, the set of $v\in\Om$ such that $v^+$ or $v^-$ is in the boundary of some properly embedded triangle. 
The collection of vectors tangent to projective lines contained in properly embedded triangles is
denoted $S\Om_{\text{flat}}$. 
These notions descend to the quotient since $\Gam$ is acting by projective
transformations, and we assign the analogous definitions to $SM_{\text{reg}}, SM_{\text{sing}},$ and
$SM_{\text{flat}}$. 
Lastly, a closed orbit $\phi\cdot v$ is \emph{hyperbolic} if when $v$ is lifted to $\tilde{v}$ in the universal
cover, $\ell_{\tilde{v}}$ is preserved by a hyperbolic group element. Note that vector with a closed
orbit which is hyperbolic must be regular (Proposition \ref{prop:nscisoms}). 
We will also denote by $d$ a Finsler metric on $SM$ compatible with the topology, see \cite[pp 161-206]{handbook}.

\subsection{Stable and unstable sets}


	Recall that the stable and unstable sets at a point are defined to be 
		\begin{align*}
			& W^{ss}(v) = \{u\in SM\mid d(\phi^tv,\phi^tu)\to 0\mid t\to+\infty\}, \\
			& W^{su}(v)= \{ u\in SM\mid d(\phi^{-t}v,\phi^{-t}u)\to 0\mid t\to+\infty\}.
		\end{align*}
		The weak stable and unstable sets of $v$ (denoted $W^{os}(v)$ and $W^{ou}(v)$,
		respectively) are the points which stay bounded distance from
	$v$ under the geodesic flow in positive and negative time, respectively. 	
	The strong stable and unstable sets are \emph{global} if for all regular $u\neq v$, at least
	one of the following are nonempty:
	$W^{ss}(v)\cap W^{ou}(u)$ or $W^{ss}(v)\cap W^{ou}(-u)$. 

\begin{proposition}
	\label{prop:globalstables}
	If $v\in SM_{\text{reg}}$ then $W^{ss}(v)$ and $W^{su}(v)$ are defined globally, and
	$W^{os}(v)$, $W^{ou}(v)$ admit a flow invariant foliation by strong stable (respectively,
	strong unstable) leaves.
\end{proposition}
	\begin{proof}
		For proper extremal points, horospheres are well-defined and the geometric
		description of stable and unstable sets applies as for the strictly convex case
		(as in \cite[Lemma 3.4]{Ben1}): that is, for 
		for $v\in S\Om_{\text{reg}}$ we have 
		\begin{align*}
			& W^{ss}(v) = \{u\in S\Om \mid \pi u\in \mathcal H_{v^+}(\pi v), \ u^+=v^+\}, \\
			& W^{su}(v)= \{ u\in S\Om\mid \pi u\in \mathcal H_{v^-}(\pi v), \ u^-=v^- \},
		\end{align*}
		where $\mathcal H_{\xi}(p)$ is the horosphere based at $\xi\in\bd\Om$ passing
		through $p\in \Om$,  
		and the strong stable and unstable sets foliate the weak stable and unstable sets:
		\begin{align*}
			W^{ou}(v) &= \bigcup_{t\in\R} W^{su}(\phi^tv) \\
				& = \{ w\in S\Om \mid w^-=v^-\},
		\end{align*}
		such that the foliation is both $\Gam$-invariant and $\tilde{\phi}^t$-invariant.
		It is then clear that for any two $u\neq v\in S\Om_{\text{reg}}$, $W^{ss}(v)\cap
		W^{ou}(u)\neq\emptyset$ as long as $u^-\neq v^+$. 
	\end{proof}
	Conversely, nonproper and nonextremal points do not have well-defined stable and unstable
	sets which foliate the weak stable and unstable sets. This can be verified by basic
	properties of the cross-ratio. 
	By Theorem \ref{thm:ben3mfldgeom}(g), the vertices of properly embedded triangles in $\Om$
	are dense in $\bd\Om$, and as such the singular points $x\in S\Om$ in $\bd\Om$ are dense in
	$S\Om$. Since these points do not admit stable and unstable sets, the geodesic flow cannot
	have local product structure.

However, the Bowen bracket for regular vectors is well-defined by 
Proposition \ref{prop:globalstables}. The \emph{Bowen-bracket} of $u$ with $v$ is the point of
intersection $w$ of the strong stable and
weak stable foliations of $u$ and $v$ such that $v^-\neq u^+$. Geometrically, $w$ is 
uniquely determined by $w^-=v^-$,
$w^+=u^+$, and $\pi w\in \mathcal H_{u^+}(\pi u)$. 

Another consequence of the geometric definition of stable and unstable sets is that the distance
between $v,u\in 
W^{os}(w)$ is monotone decreasing under the geodesic flow in positive time for an adapated metric
\cite{Cr14},
which can be verified by properties of the cross-ratio. Similarly, the distance between points in
the same unstable set is monotone decreasing under the flow in negative time. 

\subsection{Topological transitivity}

	In this subsection we prove topological transitivity, which is equivalent to existence  
	of a dense orbit when the phase space is compact
	\cite[Lemma 1.4.2]{MDS}, as in the case of the Benoist 3-manifolds. 
	 A continuous dynamical system $f^t\colon X\to X$ is \emph{topologically transitive} if for
	 every pair of open sets $U,V\subset X$, there exists a time $0<T\in \R$ such that
	 $f^T(U)\cap V\neq\emptyset$. 
	 If $X$ is a metric space then the system is \emph{uniformly transitive} if for all
	 $\delta>0$, there exists a $T>0$ such that for all $x,y\in X$, there is some $t\leq T$ such
	 that $f^t\big(B(x,\delta)\big) \cap B(y,\delta) \neq\emptyset$. 
	 \ifextended \else
  	 It is clear that transitivity implies {uniform transitivity} when $X$ is a compact metric space. 
  	 \fi

	\begin{lem} \label{lem:densepps}
	 Hyperbolic periodic points are dense for the geodesic flow of a Benoist 3-manifold. 
	\end{lem}
	
	\begin{proof}
		We want to show any $(\xi,\eta)\in\bd\Om\times\bd\Om\smallsetminus\diag$ can be
		approximated by $(g^-,g^+)$ such that $g\in\Gam_{\text{hyp}}$.		
		Take two noncommuting hyperbolic elements $g,h\in\Gam$ (Corollary \ref{cor:exist_hyp_elts}).
		Construct the sequence $k_n = g^nh^n$. Then there are fixed points $k_n^+$ and $k_n^-$ in
		$\bd\Om$ of $k_n$ such that $k_n^+ \to g^+$ and $k_{-n}^- \to h^-$ as $n\to\infty$. Using
		the sequence $k_n$ and minimality of the action of $\Gam$ on $\bd\Om$ \cite[Proposition
		3.10]{Ben4}, we conclude that any 
		$(\xi, \eta) \in \bd\Om \times \bd\Om\setminus
		\text{diag}$
		is approximable by such $k_n$. If any $k_n$ admits a projective line 
		axis, then this projective line axis corresponds to a periodic orbit for the flow and we
		conclude that any vector tangent to the projective line $(\xi\eta)$ is approximable by
		periodic points. To prove the lemma, we just need to show there necessarily exists a
		subsequence $k_{n_i}$ of only  hyperbolic elements. 
	
		By contradiction, suppose there is no such subsequence. There exists an $N$ such
		that for all $n\geq N$, each $k_n$ preserves a properly embedded triangle $\tri_n$. If we assume
		$k_n$'s geodesic axis of translation, which is not necessarily a projective line, is
		also on the triangle, we consider the accumulation of the boundary of the triangles,
		$\bd\tri_n$, in
		$\bd\Om$, which will contain $h^-$ and $g^+$. This set will be either the boundary of a
		properly embedded triangle, a line segment, or a point. None of the above
		are possible since $h,g$ are
		hyperbolic and do not commute, and $\Gam$ acts discretely so $h^-\neq g^+$ are
		proper extremal points, and $(h^-g^+) \not\subset \bd\Om$. 
	\end{proof}

	Let $\phi\cdot v$ denote the orbit of $v$.

	\begin{prop} \label{prop:toptrans}
	 The geodesic flow of  a Benoist 3-manifold is topologically transitive.
	\end{prop}
	\begin{proof}
	 Take two open sets $U$ and $V$ in $SM$. 
	 By Lemma \ref{lem:densepps}, there are hyperbolic periodic points $u\in U$ and $v\in V$. 
	 We now construct a heteroclinic orbit.
	 Lifting to the universal cover, we have $\tilde{u}\in\tilde U,\tilde{v}\in \tilde V \subset
	 S\Om$ such that $u^-$ and $v^+$ are proper extremal points of $\bd\Om$. 
	 Then 
	 the open projective line segment
	 $(u^-v^+)$ is contained in $\Om$ and is the footpoint projection of an orbit of the geodesic flow. 
	 Let $\tilde{w}\in S\Om$ denote the Bowen bracket of $\tilde v$ with $\tilde u$: 
	 \[
		 \tilde w\in W^{ss}(\tilde v)\cap W^{su}(\phi^t \tilde u)
	 \]
	 for some $t\in\R$. 
	 Since $u$, $v$ are periodic, there are hyperbolic group elements $\gam_{\tilde
	 u},\gam_{\tilde v}$ preserving
	 $\ell_{\tilde u},\ell_{\tilde v}$ so 
	 $d\gam_{\tilde u}^n(\tilde U)\cap \tilde \phi\cdot \tilde u$ and $d\gam_{\tilde v}^n(\tilde
	 V)\cap \tilde \phi\cdot \tilde v$
	 each contain lifts of $u$ and $v$ respectively for all $n\in \Z$. 
	 Since $\gam_u,\gam_v$ are isometries and $u^-=w^-,v^+=w^+$ are proper extremal points, there is
	 an $N$ such that for all $n\geq N$, $d\gam_{\tilde u}^{-n}(\tilde U)\cap \tilde \phi\cdot \tilde w \neq
	 \varnothing$ and $d\gam_{\tilde v}^n(\tilde V) \cap \tilde \phi\cdot \tilde w \neq\varnothing$. 
	 Then choosing times $t_1, t_2$ so that $\phi^{t_1}\tilde w \in d\gam_u^{-n}(\tilde U)\cap
	 \phi\cdot \tilde w$ and $\phi^{t_2}\tilde w \in d\gam_v^n(\tilde V) \cap \phi\cdot \tilde w
	 $, we can project $\phi^{t_1}\tilde{w}$ to $SM$ and obtain $T=-t_1+t_2$ such that
	 $w':=d\pi_\Gam \phi^{t_1} \tilde w \in U$, where $\pi_\Gam\colon \Om\to M$ is the quotient
	 map,  and $\phi^Tw'\in V$ as desired. 
	\end{proof}

	
	\subsection{The Anosov Closing Lemma}
	In this subsection, we prove Anosov closing of recurrent orbits, originally due to Anosov in
	the negative curvature case \cite{anosov-translation}.

	Define a filtration of $SM\smallsetminus SM_{\text{flat}} $ by compact sets bounded away from flats:
	\[
		\Lambda_{\eta}:=\left\{v\in SM \mid d(v,w)\geq \eta \text{ for all }w\in
		SM_{\text{flat}}\right\}.
	\]
	We say for points $u,v\in SM$ and $\ep>0$ that $u$ {\em $\ep$-shadows} $v$ for time
	$t$ if $d_s(u,v)<\ep$ for $s\in[0,t]$.

\begin{theorem}
	Let $\Om$ be an  indecomposable, nonstrictly convex domain in $\RP^3$. Suppose $\Gam <
	\PSL(4,\R)$ is a discrete, torsion free group dividing $\Om$, with compact quotient
	$M=\Om/\Gam$. Then for all $\eta>0$ and sufficiently small $\ep>0$, there exists a
	$\delta>0$ and $T>0$ such that: 

	For any $t\geq T$, $v\in\Lam_\eta$  with $d(\phi^tv,v)<\delta$, there exists a hyperbolic 
	periodic orbit $v'$ of period $t'\in ]t-\ep,t+\ep[$ which $\ep$-shadows $v$ for time
	$\min\{t,t'\}$.
	\label{thm:acl}
\end{theorem}
\begin{proof}
  We adapt a proof by contradiction following Eberlein \cite{eberlein} (see also \cite[Theorem
  7.1]{CouSch10}). Assume we have particular $\eta,\ep>0$ and a sequence of $v_n\in \Lam_\eta$
paired with a sequence $t_n\to\infty$ such that $d(v_n,\phi^{t_n}v_n)\to 0$, yet any $w_n$ which
$\ep$-shadows $v_n$ for time $t_n$ is not periodic of any period $t'_n\in]t_n-\ep,t_n+\ep[$.
	
	We can assume up to extraction of subsequences that the $v_n$ converge to some $v\in\Lam_\eta$. 
	Lifting $SM$ to a compact fundamental domain $SD$ containing $\tilde{v}$ in $S\Om$, we have
	some $\tilde{v}\in S\Om$ with points $v^+,v^-$ in $\bd\Om$, and lifts $\tilde{v_n}$ of the
	$v_n$ which converge to $\tilde{v}$ in $SD$. Also, since $\phi\cdot v_n$ almost closes up
	after time $t_n$, there are group elements $\gam_n$ which take $\tilde{v_n}$ close to
	$\phi^{t_n}\tilde{v_n}$. Note that the $\gam_n$ need not be hyperbolic a priori.

Again, the contradiction hypothesis is that if $w_n$ $\ep$-shadows $v_n$ for time $t_n$, then $w_n$
cannot be periodic of any period $t_n'\in]t_n-\ep,t_n+\ep[$. Eberlein's geometric observation is
that in the universal cover, if $w_n$ $\ep$-shadows $v_n$ for time $t_n$, then the same $\gam_n$
which moves $\tilde{v_n}$ close to $\phi^{t_n}\tilde{v_n}$ must also be responsible for moving
$\tilde{w_n}$ close to $\phi^{t_n}\tilde{w_n}$. Because $\Gam$ is acting on $\Om$ properly
discontinuously and cocompactly by isometries, the assumption that $w_n$ is not periodic of period
approximately $t_n$ is realized in the universal cover as follows: if
$d(\tilde{w_n},\tilde{v_n})<\ep$, then $\gam_n.\tilde{w_n} \neq \phi^{t'_n}\tilde{w_n}$ for any
$t'_n\in ]t_n-\ep,t_n+\ep[$. \\

The goal of the following lemmas will be to show that nonexistence of an axis of $\gam_n$ which is
$\ep$-close to $\ell_{\tilde{v_n}}[0,t_n]$ for infinitely many $n$ is mutually exclusive with the
assumption that the $v_n$ and $v$ are in $\Lam_{\eta}$, producing the desired contradiction.

	\begin{lemma}
		Let $p\in\Om$ be the footpoint of $\tilde{v}$. Then $\gam_n.p\to v^+$ and $\gam_n^{-1}.p\to v^-$. 
		\label{lem:ACL_lemma3}

	\end{lemma}
	\begin{proof}
		Take any convex open neighborhood $\mathcal N(v^+)$ in $\cl{\Om}$. 
		Since $\tilde{v_n}\to\tilde{v}$, we have $v_n^+\in \mathcal N(v^+)$ for all
		sufficiently large $n$. Then as $t_n\to+\infty$, $\ell_{\tilde{v_n}}(t_n) 
		\in \mathcal N(v^+)$ by convexity of $\mathcal N(v^+)$. 
		Since $\gam_n$ is chosen so that $d(\gam_n.\tilde{v_n},\phi^{t_n}\tilde{v_n})\to 0$
		as $n\to\infty$, then $d_\Om(\gam_n.p_n, \ell_{\tilde{v_n}}(t_n))\to0$ with $n$.
		Once $\ell_{\tilde{v_n}}\in\mathcal N(v^+)$ for all large enough $n$ and
		$\gam_n.p_n$ is sufficiently close to $\ell_{\tilde{v_n}}$, we will have
		$\gam_n.p_n\in\mathcal N(v^+)$. 
		Finally, as $\tilde{v_n}\to\tilde{v}$ implies $p_n\to p$ and $\gam_n$ is an
		isometry, we can conclude for large $n$ that $\gam_n.p\in\mathcal N(v^+)$.

		Now consider $\mathcal N(v^-)$ a convex open neighborhood of $v^-$ in $\cl{\Om}$.
		As $\gam_n.\tilde{v_n}$ approaches $\phi^{t_n}\tilde{v_n}$, the group element
		$\gam_n^{-1}$ brings the line segment $\ell_{\tilde{v_n}}[-s_n+t_n,s_n+t_n]$ back
		very close to the line segment $\ell_{\tilde{v_n}}[-s_n,s_n]$ for some
		$s_n\to\infty$ with $n$. Then as $s_n$ gets very large, $\ell_{\tilde{v_n}}(-s_n)$
		gets closer to $v_n^-$, as will $\gam_n.\ell_{\tilde{v_n}}(-s_n+t_n)$ which is
		converging to $\gam_n.v_n^-$ with large $s_n$. Hence, $\gam^{-1}_n.v_n^-$ approaches
		$v_n^-$ in the boundary. Then as $\tilde{v_n}\to\tilde{v}$, for sufficiently large
		$n$, we have $\gam^{-1}_n.v^-\in \mathcal N(v^-)$. Since $\gam_n^{-1}.p_n$ is a point on the
		line $\gam^{-1}.\ell_{\tilde{v_n}}$, it suffices to observe that
		$d_\Om(\gam^{-1}_n.p_n,p_n)=d_\Om(p_n,\gam_n.p_n)\sim t_n\to\infty$ as $n\to\infty$
		to conclude $\gam^{-1}_n.p \in \mathcal N(v^-)$ for all sufficiently large $n$. 
	\end{proof}

	We next define  $V_k(v^+), \; V_k(v^-)$ open neighborhoods in $\bd\Om$ such that for any
	$\xi\in V_k(v^+), \; \zeta\in V_k(v^-)$, the projective line $(\zeta\xi)$ is distance less
	than $\frac1k$ from $\ell_v (0)$ in the Hilbert metric.  The existence of such $V_k$ is
	immediate in a Hilbert geometry by the definition of $d_\Omega$. 
	The $V_k$ are also homeomorphic to open balls in $\R^2$. Choose $k$ large enough that $\frac1k<\frac\ep2$.\\

	\begin{lemma}
		For all sufficiently large $n$, $\gam_n(\cl{V_k(v^+)})\subset {V_k(v^+)}$ and
		$\gam^{-1}_n(\cl{V_k(v^-)})\subset {V_k(v^-)}$. 

		\label{lem:ACL_lemma4}
	\end{lemma}
	\begin{proof}
		Note that as $\gam_n.v_n^+$ is closer to $v_n^+$ and $\tilde{v_n}\to\tilde{v}$, then 
		$\gam_n.v^+\to v^+$ (and similarly, $\gam_n^{-1}.v^-\to v^-$). 
		If $\gam_n.v^+$ is very close to $v^+$, then $\gam_n$ either fixes $v^+$, is
		contracting near $v^+$, or both. The only way that $\gam_n(\cl{V_k(v^+)})\not\subset
		V_k(v^+)$ would be if $\gam_n$ stabilized a properly embedded triangle $\tri_n$ such
		that $\bd\tri_n \cap \bd V_k(v^+) \neq\varnothing$. If this happened infinitely
		often, then $v^+$ would necessarily be the limit of vertices of $\tri_n$ which are
		attracting eigenvectors for the $\gam_n\in \Stab(\tri_n)$. 
		Since $\gam_n^{-1}.v^-\to v^-$ and $\gam_n^{-1}\in\Stab(\tri_n)$, we can also
		conclude that vertices of $\tri_n$ which are repelling eigenvectors for $\gam_n$
		must accumulate on $v^-$. Then in the quotient $SM$, for large enough $n$, $v$ must
		be distance less than $\eta$ from a quotient torus of one of the $\tri_n$,
		contradicting the assumption that $v\in \Lam_\eta$ for small $\ep$. 

		An analogous argument applies to show, up to extraction of subsequences, for all
		sufficiently large $n$, $\gam_n^{-1}(\cl{V_k(v^-)})\subset V_k(v^-)$. 
	\end{proof}

	So we now have that for large $n$, $\gam_n(\cl{V_k(v^+)})\subset V_k(v^+)$ and similarly
	$\gam^{-1}_n(\cl{V_k(v^-)})\subset V_k(v^-)$, both of which are homeomorphic to open balls
	in $\R^2$.
	 Applying Brouwer's fixed point theorem,
	 it follows that $\gam_n$ fixes points in $V_k(v^-)$ and $V_k(v^+)$. Then $\gam_n$ has an
	 axis distance less than $\frac1k<\frac\ep2$ from $\ell_{\tilde{v}}(0)$, hence $\ep$-close
	 to $\ell_{\tilde{v_n}}(0)$ for all sufficiently large $n$. We also assume that
	 $\gam_n.\tilde{v_n}$ is arbitrarily close to $\phi^{t_n}\tilde{v_n}$, so the axis of
	 $\gam_n$ will eventually and thereafter be $\ep$-close to $\ell_{\tilde{v_n}}[0,t_n]$ and
	 the translation length of $\gam_n$ must be $\ep$-close to $t_n$. And so we have a periodic
	 orbit of period $t_n'\in]t_n-\ep,t_n+\ep[$ which $\ep$-shadows $v_n$ for time
		 $\max\{t_n,t_n'\}$, 
	 contradicting the assumption.  
	 If we obey our hypothesis that such a periodic orbit is impossible, then we would
	 necessarily have $v\not\in \Lam_\eta$ as proven in Lemma \ref{lem:ACL_lemma4} -- a
	 contradiction. \\

	 Lastly, note that for small $\ep$, hyperbolicity of the periodic orbit is implicit, since a
	 periodic orbit tangent to a torus is bounded away from $v_n$ by $\eta$ and thus could not
	 $\ep$-shadow $v_n$ for small $\ep$. 
\end{proof}

\section{Topological mixing}
\label{sec:mixing}

	We prove the geodesic flow of a Benoist 3-manifold is topologically mixing following the
	strategies of Coudene \cite{Cou04}, but without the local product structure property. 
	Key properties will be a nonuniform orbit gluing lemma (Lemma \ref{lem:orbitglue}) and density of
	the unstable leaves for periodic regular vectors (Proposition \ref{prop:denseunstables}). 

Let $W^\ast_\ep(v)=W^\ast(v)\cap B(v,\ep)$ and let $\langle v,u\rangle$ denote the Bowen
bracket of $v$ with $u$ where $v,u$ are regular vectors.  
\begin{proposition}
	\label{prop:bowenbracket}
	For all $\ep>0$ and $u\in SM_{\text{reg}}$, there is a $\delta>0$ such that for all $v\in
	B(u,\delta)\cap SM_{\text{reg}}$, there exists a $|t|<\ep$ such that 
	\[
		\langle v,u\rangle \in W^{su}_\ep(\phi^tv)\cap W^{ss}_\ep(u).
	\]
\end{proposition}
\begin{proof}
	Lift $u$ to $\tilde u\in S\Om$, with $\pi \tilde{u}\in \text{int}(D)$ a fundamental domain for the $\Gam$-action on $\Om$.
	For all $\ep>0$, there are neighborhoods $U^+$ of $u^+, U^-$ of $ u^-$ in $\bd\Om$ such that
	$v\in B(u,\ep)\then v^+\in U^+, v^-\in U^-$. 
	For any such neighborhoods, if $v$ is such that $v^+\in U^+,v^-\in U^-$ and $\pi v$ is
	sufficiently close to $\pi u$, then $v\in B(u,\ep)$. 
	Make $U^-$ small enough that $U^- \subset \{ w^- \in \bd\Om \mid w\in W_\ep^{ss}(u)\}$
	guarantees that any $v$ with $v^-\in U^-$ satisfies $\langle v,u \rangle \in
	W^{ss}_\ep(u)$.
	Taking $U^+$ be as small as needed, we can make these $v$ with $v^-\in U^-$ 
	arbitrarily close to $\phi^{t_v}\langle v,u\rangle$ in this local neighborhood of $u$. It
	suffices to choose $\delta>0$ sufficiently small as to ensure $|t_v|<\ep$. 
\end{proof}

	\subsection{Orbit gluing in Hilbert geometries}

Uniform orbit gluing is also known as shadowing of pseudo-orbits in the literature. 
We introduce a weaker notion here. 
We can associate to any orbit segment $\phi^{[0,t]}v$ the pair $(v,t)\in SM\times \R^+_0$. 
An $n$-length $\delta$-\emph{pseudo-orbit} is a collection of $n$-many finite length orbit segments
$\{(v_i,t_i)\}_{i=1}^n\subset SM \times \R^+_0$ such that $d(\phi^{t_i}v_i,v_{i+1})<\delta$ for
$i=1,\ldots, n-1$. 
	\begin{definition}
		The dynamics satisfies \emph{weak orbit gluing} if for all $\ep>0$ and
		$\{v_i\}_{i=1}^n$ there exists $\delta >0$ such that for all $n$-length
		$\delta$-pseudo orbits $\{(v_i,t_i)\}$
		 there is a point $w$ which $\ep$-shadows the orbit segments
		 $[v_1,t_1],\ldots,[v_{n-1}, t_{n-1}], [ v_n, +\infty]$. More explicitly: for some 
			$|t|<\ep$,
		 \[
		   w \in W^{su}_\ep(\phi^tv_1) \quad \text{and} \quad
			 \phi^{\sum t_i}w \in W^{ss}_\ep(v_n),
		 \]
		and there are numbers $|t_i'|<\ep$ for $i=1,\ldots, n-2$ such that for all $k= 2,\ldots, n-1$, 
		\begin{align*}
			d(\phi^{t_1+\cdots + t_{k-1} +s }(w), \phi^{t_k' + s}(v_k)) < \ep, & \qquad
			\text{if }0<s<t_k, \\
			d(\phi^s(w), \phi^{t_1'+s}(v_1))<\ep, &\qquad \text{if } 0< s< t_1.
		\end{align*} 
	\end{definition}

\begin{lemma}[weak orbit gluing]
 \label{lem:orbitglue} 
 The geodesic flow of the Benoist 3-manifolds satisfies weak orbit gluing for pseudo-orbits
 $\{(v_i,t_i)\}_{i=1}^n$ such that $v_1,\ldots,v_{n-1}$ are backward regular and $v_n$ is forward regular.
\end{lemma}

\begin{proof}
	The proof is effectively a finite recursive application of taking Bowen brackets
	(Proposition \ref{prop:bowenbracket}). Suppose $d(\phi^{t_i}v_i,v_{i+1})<\delta_1$ for all
	$i=1,\ldots,n-1$. 
	For sufficiently small $\delta_1>0$, if $\phi^{t_1}v_1\in B(v_2,\delta_1)$
	then the Bowen bracket
	$w_1$ is in $W_{\delta_2}^{su}(\phi^{t_1+t_1'}v_1)\cap W_{\delta_2}^{ss}(v_2)$ for some
	$|t_1'|<\delta_2$. Then 
	$\phi^{t_2}w_1\in B(v_3,\delta_2+\delta_1)$ and we will have $w_2\in
	W_{\delta_2}^{su}(\phi^{t_2+t_2'}w_1)\cap W^{ss}_{\delta_3}(v_3)$ for $|t_2'|<\delta_3$ if
	$\delta_1,\delta_2$ are sufficiently small. Repeating the argument, we have
	$\phi^{t_3}w_2\in B(v_4,\delta_3+\delta_1)$ implying there exists $w_3\in
	W_{\delta_4}^{su}(\phi^{t_3+t_3'}w_2)\cap W_{\delta_4}^{ss}(v_4)$ and so on, until we find
	$\phi^{t_{n-1}}w_{n-2}\in B(v_n,\delta_{n-1}+\delta_1)$ gives 
	$w_{n-1}\in W_{\delta_n}^{su}(\phi^{t_n+t_{n-1}'}w_{n-2})\cap W_{\delta_n}^{ss}(v_n)$ for 
	$|t_{n-1}'|<\delta_n$.
	Observe the following:
	\begin{align}
		 & w_k\in W_{d_{k+1}}^{su}(\phi^{t_k+t_k'}w_{k-1}) \text{ for all }k=2,\ldots,n-1, 
		\label{eqn:ik} \\
		 & w_k\in W_{\delta_{k+1}}^{ss}(v_{k+1}) \text{ for all }k=1,\ldots,n-1,
		\label{eqn:iik} \\
		& w_1 \in W^{su}_{\delta_2}(\phi^{t_1+t_1'}v_1), 
		\label{eqn:iv}
		\\
		& \delta_{k} \text{ depends only on }v_{k+1} \text{ and }\delta_{k+1},
		\label{eqn:iiik}
		\\
		& \qquad \qquad \qquad \qquad \qquad \qquad \text{ and } |t_k'|<\delta_{k+1} \text{
		for }k=2,\ldots, n-1. 
		\nonumber
	\end{align}
	Though $\delta_1$ depends on $\delta_2,\ldots, \delta_n$ and $v_2,\ldots,v_n$, this is still a
	finite amount of data. We will also need to make $\delta_1$ smaller to meet the conditions
	of weak orbit gluing, which we now address.  

	Let $w\in \phi^{-\sum_{i=1}^{n-1}t_i}w_{n-1}$. If $\delta_n<\epsilon$, we have
	$\phi^{\sum_{i=1}^{n-1}t_i}w\in W^{ss}_\ep(v_n)$ immediately. 
	Moreover, for $k=2,\ldots,n-1$ and $s\in [0,t_k]$, 
	\begin{align*}
		& \hspace{-5em} d(\phi^{t_1+\cdots+t_{k-1}+s}w,\phi^{t_k'+\cdots+t_{n-1}'+s}v_k) \\
		& = d(\phi^{-t_k-\cdots - t_{n-1}+s}w_{n-1},\phi^{t_k'+\cdots+t_{n-1}'+s}v_k) \\
		&\leq
		d(\phi^{-t_k-\cdots-t_{n-1}+s}w_{n-1},\phi^{-t_k-\cdots-t_{n-2}+t'_{n-1}+s}w_{n-2})
		\\
		& \qquad +
		d(\phi^{-t_k-\cdots-t_{n-2}+s}w_{n-2},\phi^{-t_k-\cdots-t_{n-3}+t_{n-1}'+t_{n-2}'+s}w_{n-3})
		\\
		& \qquad\qquad + \cdots + d(\phi^{-t_k+t_{k+!}'+\cdots + t_{n-1}'+s}w_k,
		\phi^{t_k'+t_{k+1}'+\cdots+t_{n-1}'}w_{k-1}) \\
		& \qquad\qquad\qquad + d(\phi^{t_k'+\cdots
			+t_{n-1}'+s}w_{k-1},\phi^{t_k'+\cdots+t_{n-1}'+s}v_k)
			 < \sum_{i=k}^n \delta_i + 2 \sum_{i=k}^{n-1}|t_i'|
	\end{align*}
	by Equation \eqref{eqn:ik} for $k+1,\ldots, n-1,n$ and Equation \eqref{eqn:iik} for $k-1$. 
	Similarly, for $s\in[0,t_1]$, with the addition of Equation \eqref{eqn:iv}, 
	\begin{align*}
		& \hspace{-5em} d(\phi^sw,\phi^{t_1'+\cdots +t_{n-1}'+s}v_1)  \leq \sum_{i=2}^{n-1}\delta_i +
		d(\phi^{-t_1+t_2'+\cdots +t_{n-1}'+s}w_1,\phi^{t_1'+\cdots +t_{n-1}'}v_1) \\
		& < \sum_{i=2}^{n-1}\delta_i + 2\sum_{i=2}^{n-1}|t_i'| +
		d(\phi^{-t_1+s},\phi^{t_1'+s}v_1) 
		 < \sum_{i=2}^{n-1}\delta_i + 2\sum_{i=2}^{n-1}|t_i'| + \delta_2.
	\end{align*}
	Then since $w\in W^{ou}(v_1)$ is clear, the $\delta_i, |t_i'|$ can be made sufficiently
	small to meet the definition of weak orbit gluing by the remark in Equation \eqref{eqn:iiik}. 
\end{proof}

\subsection{Density of unstable sets}

Using Proposition \ref{prop:densespec}, that the additive subgroup generated by translation lengths
of closed
hyperbolic orbits is dense in $\R$, 
we now show that unstable sets for periodic points are dense and shortly thereafter conclude the geodesic
flow is topologically mixing. 
Let $\mathcal P$ be the set of periodic points of the geodesic flow up to orbit equivalence and let $\mathcal
P_{\text{hyp}}$ be all the hyperbolic periodic points in $\mathcal P$. 
Let $T_p$ denote the length of the orbit of $p\in\mathcal P$. Note that $T_p = \tau(\gam_{\tilde{p}})$ where
$\gam_{\tilde{p}}\in\Gam$ is the 
hyperbolic isometry which perserves the projective line $\ell_{\tilde{p}}$ in $\Omega$ 
determined by some lift $\tilde{p}$ of $p$. 

The following lemma uses transitivity (Proposition \ref{prop:toptrans}), Anosov Closing (Theorem
\ref{thm:acl}), Orbit gluing of 3 orbit segments (Lemma \ref{lem:orbitglue}), and density of
$\langle T_p\rangle_{p\in\mathcal P_{\text{hyp}}}$ in $\R$
(Proposition \ref{prop:densespec}). 

\begin{lemma} \label{lem:noncomm} 
For all open $U\subset SM$, the lengths of periodic orbits passing through $U$ generate a dense subgroup of $\R$. 
\end{lemma}

\begin{proof}
	Let $p\in\mathcal P_{\text{hyp}}$.
	Since $SM_{\text{flat}}$ is closed, it suffices to assume $U\cap SM_{\text{flat}}=\varnothing$.
	Choose $\eta>0$ such that $U\subset \Lambda_\eta$ (recall that $\Lambda_\eta$ is the
	compliment of the $\eta$-neighborhood of the flats in $SM$). 
	Let $\ep>0$, and 
	consider a point $v_0\in U$ with a dense forward orbit, with $\ep$ small enough that
	$B(v_0,2\ep)\subset U$. 
	Choose $0<\delta(\ep,\eta)<\ep$ small enough to satisfy Anosov Closing on
	$\Lambda_\eta$. 
	Choose $\delta'(\eta, \frac\delta6,3)<\frac\delta3$ as for $\frac\delta6$-fine orbit gluing
	for 3 orbit segments with starting points $v_0,p,v_0$.  
	Then there exist $s_0,s_1>0$ such that $d(\phi^{s_0}v_0,p)<\delta'$ and
	$d(\phi^{s_0+s_1}v_0,v_0)<\delta'$ by transitivity of $v_0$. 
	Thus, the orbit segments $\{(v_0,s_0), (p,nT_p),(\phi^{s_0}v_0,s_1) \}$ can by glued by some
	$v_n$ with fineness $\frac\delta6$:
	\begin{align*}
		& v_n\in B_{s_0}\big(\phi^{t_1'}v_0,\frac\delta6\big) \text{ for some }
		|t_1'|<\frac\delta6, \\
		& \phi^{s_0}v_n \in B_{nT_p}\big(\phi^{t_2'}p,\frac\delta6\big) \text{ for some }
		|t_2'|<\frac\delta6, \\
		& \phi^{s_0+T_p}v_n\in B_{s_1}\big(\phi^{t_3'}\phi^{s_0}v_0,\frac\delta6\big)
		\text{ for some }|t_3'|<\frac\delta6.
	\end{align*}
	Then 
	\begin{align*}
			d(v_n,\phi^{s_0+nT_p+s_1}(v_n)) & \leq
			d(v_n,v_0)+d(v_0,\phi^{s_0+s_1}v_0)+d(\phi^{s_0+s_1}v_0,\phi^{s_0+nT_p+s_1}v_n)
			\\ 
			& < 2\left( \frac\delta6 \right) + \left( \frac\delta3 \right) + 2\left(
			\frac\delta6 \right)= \delta.
		\end{align*}
	Note that $v_n\in B(v_0,\delta/3)\subset U \subset \Lambda_\eta$. 
	By Anosov closing on $\Lambda_\eta$, there exists a regular $w_n$ which has period length
	$s_0+nT_p+s_1+t_n'$ for $|t_n'|<\ep$. 
	Since $w_n$ also $\ep$-shadows $v_n$, we have $w_n\in B(v_0,2\ep)\subset U$. 
	We can repeat the above argument for all $n$ with the same $\eta, \ep, p$ and $v_0$, hence
	the same $\delta,\delta'$ and the same $s_0,s_1$. Then we have $w_n,w_{n+1}\in U$ and thus
	\[
		{\langle T_q\rangle}_{q\in U\cap \mathcal P_{\text{hyp}}} \ni s_0+(n+1)T_p+s_1 +
		t_{n+1}' - (s_0+nT_p+s_1+t_n') = T_p + t_{n+1}'-t_n'.
	\]
	Since $|t_{n+1}'-t_n'|\leq |t_{n+1}'|+|t_n'|<2\ep$, letting $\ep$ go to zero we conclude
	$T_p\in\cl{\langle T_q\rangle}_{q\in U\cap \mathcal P_{\text{hyp}}}$ for all $p\in\mathcal
	P_{\text{hyp}}$, which proves the lemma because $\cl{\langle T_p\rangle }_{p\in\mathcal
		P_{\text{hyp}}}=\R$. 
\end{proof}

We are now prepared to prove a key proposition.

\begin{prop}
	If $v\in SM$ is a hyperbolic periodic orbit, $W^{su}(v)$ is dense in $SM$.
\label{prop:denseunstables}
\end{prop}
\begin{proof}
	Let $U\subset SM$  be open. 
	By Lemma \ref{lem:noncomm} there exists a $u\in\mathcal
	P_{\text{hyp}}\cap U$ such that $\cl{\langle T_v,T_u\rangle} = \R$. 
	Let $\ep>0$ such that $B(u,\ep)\subset U$. 
	Since $v,u\in SM_{\text{reg}}$ there exists a $T\in\R$ such that $w\in W^{su}(v)\cap W^{ss}(\phi^T(u))$. 
	Then $\phi^{-T}w\in W^{ss}(u)$ so choose $M\in \N$ large enough that for any $m\geq M$,
	$d(\phi^{mT_u-T}w,u)<\ep/2$. 
	Because $\cl{\langle T_v,T_u\rangle} = \R$, there are large enough $m,n\in\N$ with $m\geq M$ such that 
	\[
	-\ep/2 < |-nT_v+mT_u-T|<\ep/2
	\]
	implying that $d(\phi^{-nT_v}w,\phi^{mT_u-T}w)<\ep/2$ and hence $\phi^{-nT_v}w\in B(u,\ep)\subset U$.  
	Note that
	$\phi^{-nT_v}w\in\phi^{-nT_v}W^{su}(v)=W^{su}(v)$ to conclude the proof. 
\end{proof}

The following lemma, the final piece preceding the proof of topological mixing, is generally taken
as fact. 
We have included the proof for completeness. 

\begin{lemma} \label{lem:stablemflddensity_altdefn}
	Let $f^t\colon X\to X$ be any continuous flow of a compact metric space. 
For $p$ periodic, 
density of $W^{su}(p)$  implies that for all $\delta>0$ and for all $x\in X$, 
there is a $T(p,\delta, x)>0$ such that for all $t\geq T$, 
$f^t W^{su}_\delta(p) \cap B(x,\delta) \neq \emptyset$.
\end{lemma}
\begin{proof}
	By assumption, there exists some $z\in W^{su}(p) \cap B(x,\delta/2) $. Then
	$d(f^{-t}z,f^{-t}p)\to 0$ as $t\to+\infty$ so there exists an $S>0$ such that $s \geq
	S$ implies $ d(f^{-s}p, f^{-s}z)<\delta$. 
	For all $n\in \N$ such that $nT_p \geq S$,
	then $d(p,f^{-nT_p}z) = d(f^{-nT_p}p, f^{-nT_p}z) < \delta$, and $f^{-nT_p}z \in
	f^{-nT_p}W^{su}(p) = W^{su}(p)$.
	Hence $f^{-nT_p}(z) \in W^{su}_{\delta}(p)$ and $z\in f^{nT_p}(W_{\delta}^{su}(p))\cap
	B(x,{\delta/2}) \neq \emptyset$. 

	Take a finite $\delta/2$-cover $\{t_1,\ldots, t_k\}$ of $[0, T_p]$. 
	Repeat for each periodic point $f^{t_i}p$ of period $T_p$ the above argument to produce a
	$z_i\in W^{su}(f^{t_i}p)\cap B(x,\delta/2)$ and minimum $n_i\in \N$ such that if $n\geq n_i
	$, then   
	$z_i\in f^{nT_p}(W_{\delta}^{su}(f^{t_i}p))\cap B(x,{\delta/2}) \subset
	f^{nT_p+t_i}(W_{\delta}^{su}(p))\cap B(x,\delta/2) \neq \emptyset$. 
	Let $N = \max_{1\leq i \leq k } n_i$ and $T = (N+1)T_p$. Then for all $t\geq T$, there is
	some $M_t\geq N+1$, $i\in\{1,\ldots,k\}$, and $0\leq \ep\leq\delta/2$ such that $t= M_t T_p
	+ t_i + \ep$ and thus 
	\[ 
		z_i\in f^{M_tT_p + t_i} (W^{su}_{\delta} (p) )\cap B(x,{\delta/2}) 
	\]
	so $f^{\ep}z_i \in f^t(W^{su}_{\delta} (p)) \cap B(x,{\delta}) \neq \emptyset$ as desired. 
\end{proof}

We are now ready to prove the main theorem of the section. 

\begin{theorem} 
	The geodesic flow on $SM$ is topologically mixing. 
\label{thm:topologicalmixing}
\end{theorem}
\begin{proof}
  Let $U,V$ be open in $SM$ and $p\in U$ a hyperbolic periodic point. 
	Let $\delta>0 $ be small enough that $W^{su}_\delta(p)\subset B(p,\delta)\subset U$ and
	$B(x,\delta)\subset V$ for some $x\in V$. 
	By Lemma \ref{lem:stablemflddensity_altdefn}, a consequence of Proposition
	\ref{prop:denseunstables}, there is a $T(U,V)>0$ such that for all $t\geq T$, 
	\[
		\varnothing \neq \phi^t W^{su}_\delta(p) \cap B(x,\delta) \subset \phi^t U \cap V.
	\]
\end{proof}

\section{Entropy-expansiveness of Hilbert geometries}
\label{sec:hexp}

	In this section, we prove entropy-expansiveness for the geodesic flow of any compact Hilbert
	geometry. 
	First, we review some preliminary notions from entropy theory. 
	Given any metric space admitting a flow, one can define the \emph{Bowen distance} by
	\[
		d_t(v,u) := \max_{0\leq s\leq t} d(\phi^sv, \phi^su).
	\]
 	Then $d_t$ is a metric on $SM$, nondecreasing in $t$. Metric $d_t$-balls are called
	\emph{Bowen balls}, denoted $B_t(v,\delta)$.
	A \emph{$(t,\delta)$-spanning set} for $K\subset SM$ is one which is $\delta$-dense in $K$
	for the $d_t$ metric. For any compact $K\subseteq SM$, we can choose a minimal finite
	$(t,\delta)$-spanning set and denote the cardinality by $S(t,\delta, K)$. Then we define the
	\emph{topological entropy} of $\phi^t$ on $K$ by 
 \[
	 h_{top} (\phi, K) := \lim_{\delta\to 0} \lim_{t\to\infty}\frac1t \log S(t,\delta, K).
 \]
 
 There are many equivalent definitions of $h_{top}$ \cite
 {MDS}, and we include one other here. For
 $K\subseteq SM$ compact, we define a $(t,\delta)$-separated set $R\subset K$ such that for all
 $u,v\in R$, $d_t(v,u)\geq \delta$. Let $R(t,\delta,K)$ denote the maximal cardinality for
 $(t,\delta)$-separated sets, which is again finite by compactness of $K$. Then 
 \[
	 h_{top} (\phi, K) = \lim_{\delta\to 0} \lim_{t\to\infty}\frac1t \log R(t,\delta, K).
 \] 
 
 When $K=SM$, we abbreviate $S(t,\ep):=S(t,\ep,SM)$, $R(t,\ep):=R(t,\ep,SM)$, and
 $h_{top}(\phi):=h_{top}(\phi,SM)$. 

 For the purposes of applying Bowen's work, we take 
	\[
		h_{top}(\phi, K, \delta) := \lim_{t\to\infty}\frac1t \log S(t,\delta,K).
	\] 
	so that $h_{top}(\phi,K)=\displaystyle\lim_{\delta\to 0} h_{top}(\phi,K,\delta)$, and for
	$K=SM$ we have $h_{top}(\phi) := \displaystyle\lim_{\delta\to0} h_{top}(\phi,\delta)$. 

	For any point $v$ in $SM$, we define the \emph{infinite Bowen balls} about $v$ in positive
	or negative time:
\[
	\Phi_\ep(v) := \bigcap_{t\in \R^+} \phi^{-t} \cl{B( \phi^t v,\ep)} =  \{ y\in M \mid
		d(\phi^ty, \phi^tv)\leq \ep \text{ for all } t\in\R^+\}.
\]
Intuitively, we should think of the $\Phi_\ep(v)$ as the exceptions to expansivity. 
An expansive map (not flow) is defined by the existence of an $\ep>0$ such that $\Phi_\ep(v) =
\{v\}$ for all $v$. An expansive flow would satisfy that $\Phi_\ep(v)=W^{os}_\ep(v)$ for all $v$. 
There are special cases of entropy expansive systems. 
Define 
\[ 
	h^\ast(\ep) := \sup_{v\in SM} h_{top}(\phi,\Phi_\ep(v)).
\]
Then $\phi$ is \emph{$h$-expansive} with expansivity constant $\ep>0$ if $h^\ast(\ep)=0$.
In other words, there is an $\ep>0$ such that the exceptions to $\ep$-expansivity have no influence
on the entropy of the system.
For an $h$-expansive system, Bowen proved that we can bypass the cumbersome limit over $\delta\to0$
of $h_{top}(\phi,\delta)$ to compute $h_{top}(\phi)$. 

	\begin{theorem}[{\cite[Theorem 2.4]{bowen72}}]
		\label{thm:bowenhexp}
		If $\ep$ is an $h$-expansive constant for $\phi$, then 
		\[ 
			h_{top}(\phi)=h_{top}(\phi,\ep).
		\]
	\end{theorem}

	Moreover, to compute the metric entropy of a system, one can simply take a
	sufficiently fine measurable partition rather than an infimum over all possible partitions.
	An immediate consequence is existence of a measure of maximal entropy
	(see \cite[Theorem 8.6 (2)]{Wa82}).

	For any manifold, the \emph{injectivity radius} of $x\in M$ is defined to be 
	\[
	\inj(x) := \frac12 \inf_{\ell}\{\text{length}(\ell)\},
	\]
	where $\ell$ varies over all homotopically nontrivial loops through $x$. Then define the
	\emph{injectivity radius of} $M$ to be
	\[
		\inj(M) := \inf_{x\in M}\inj(x).
	\]
	If $M$ is compact then $\inj(M)>0$.

	\begin{theorem} \label{thm:hexp}
	The geodesic flow $\phi^t$ on any compact Hilbert geometry is $h$-expansive.
	\end{theorem}
	\begin{proof}
		Lift $v$ to $\tilde{v}$ in $S\Om$. If $\tilde{v}^+$ is extremal, then by properties
		of the Hilbert metric, 
		$\tilde{u}^+\neq \tilde{v}^+$ for any lift $\tilde u$  of $u$ implies $u\not\in \Phi_\ep(v)$ for
		$0<\ep<\inj(M)/3$. Then a $(0,\delta)$-spanning set for $\Phi_\ep(v)$ is a
		$(t,\delta)$-spanning set for all $t>0$ and $h_{top}(\Phi_\ep(v))=0$. 

		Suppose now that $v^+$ is not extremal. 
		Let $C\subset \bd\Om$ be the intersection of all supporting hyperplanes to $\Om$ at
		$\tilde v^+$. Note that if $\tilde v^+$ is not extremal then $C$ has nonempty
		interior for the subspace topology in the minimal projective subspace containing $C$. Since
		$C$ is properly convex in this projective subspace, 
		we can extend the Hilbert metric to the interior of $C$, which we will denote
		$d_C$, with metric balls denoted by $B_C$. 
		Now define
		\[
			\Phi_C^+(\tilde v,\ep) := \{u\in \cl{B(\tilde v,\ep)} \mid u^+\in \cl{B_C(\tilde v^+,\ep)} \}.
		\]
		Then $\Phi_\ep(v) $ is contained in the quotient projection of $\Phi_C^+(\tilde v,\ep)$.
		For all $\eta\in B_C(\tilde v^+,\ep)$ let $v_\eta$ be such that $\pi v_\eta =\pi
		\tilde v$
		and $v_\eta^+=\eta$. 
		Then $d(v_\eta,v)\leq d_C(\eta,\tilde v^+)\leq \ep$. 
		Then for all $w\in \Phi_C^+(\tilde v,\ep)$, there is an $\eta=w^+$ implying
		$d(w,v_\eta)\leq d(w,\tilde v)+d(\tilde v,v_\eta)=\ep+\ep=2\ep$, hence
		\[
			\Phi_C^+(\tilde v,\ep)\subset \bigcup_{\eta\in\cl{B_C(\tilde v^+,\ep)}} \Phi^+(v_\eta,2\ep).
		\]
		Choose a finite $\delta/2$-cover of $B_C(\tilde v^+,\ep/2)$ by $\{\eta_i\}_{i=1}^k$
		and $v_i:= v_{\eta_i}$. 
		Then for all $u\in \Phi^+(v_\eta,2\ep)$, there is an $\eta_i$ such that
		$d_C(\eta,\eta_i)<\delta/2$ and 
		$d(u,v_i)\leq d(u,v_\eta)+d(v_\eta,v_i)<2\ep+d_C(\eta,\eta_i)<\frac{5\ep}2 $
		for $\delta$ small.
		Describe all such $u$ by
		\[
			\Phi_C^+\left( v_i,{5\ep}/2,\delta/2 \right) := \{u\in
				\cl{B(\tilde v,5\ep/2)} \mid u^+\in \cl{B_C(\tilde v^+,\delta/2)} \}.
		\]
		Then for $\ep<\inj(M)/3$, 
		\[
			\Phi_\ep(v)\subset \bigcup_{i=1}^k \Phi_C^+\left(
			d\pi_\Gam v_i,{5\ep}/2,\delta/2 \right) 
		\]
		Note that for each compact $\Phi_C^+(d\pi_\Gam v_i,{5\ep}/{2},\delta/2)$, a minimal
		$(0,\delta)$-spanning set $E_i$ will be a $(t,\delta)$-spanning set for all
		$t\geq 0$. Thus,
		\[
			h_{top}(\Phi_\ep(v),\delta) \leq \lim_{t\to\infty}\frac1t \log
			\left( \sum_{i=1}^k |E_i| \right) =0.
		\]
	\end{proof}

	\section{Applications to counting and entropy}
	\label{sec:counting}

	Let $P_t(\phi)$ denote the collection of \emph{isolated} $\phi$-periodic
	orbits of period at most $t $, modulo orbit equivalence, and 
	\[
		\rho(\phi): = \lim_{t\to\infty} \frac1t  \log \# P_t(\phi).
	\]
	
	The effect of defining $P_t$ by isolated orbits is that we neglect the flat orbits in
	$SM_{\text{flat}}$. 
	Any periodic orbit on a flat corresponds to a continuous 1-parameter family
	of periodic orbits of the same homotopy type, so these are not isolated and
	not counted. 
	By contrast, the hyperbolic periodic orbits are isolated and countable. 

	The next proposition follows from $h$-expansivity (Theorem
	\ref{thm:hexp}). 
	
	\begin{prop}
		\label{prop:entropyineq}
	The geodesic flow $\phi^t$ of a Benoist 3-manifold satisfies  
		\[\rho (\phi) \leq h_{top}(\phi).\]
	\end{prop}
	\begin{proof}
		Choose $\ep\leq \inj M/3$ an $h$-expansivity constant for the geodesic flow on $SM$. 
		We show that $P_t$ is a $(t,\ep)$-separated set. 
		If $v,u\in P_T$ such that $d_T(v,u)<\ep$, then $d_t(v,u) <\ep$ for all $t\in \R$. 
		Since $\Gam$ acts discretely and $\ep<\inj(M)/3$, this is only possible if $v=u$ or
		if $v$ and $u$ lift to tangent vectors to a properly embedded triangle $\tri$ such
		that $\ell_{\tilde{v}},\ell_{\tilde{u}}\subset\tri$. Then $v,u$ are in a flat 
		so they are
		not counted in $P_T$.
	
		Thus, $P_t$ is $(t,\ep)$-separated and has cardinality at most 
		$R(t,\ep)$, the cardinality of a maximal $(t,\ep)$-separated set. We conclude by
		$h$-expansivity and Bowen's Theorem \ref{thm:bowenhexp} that 
		\[
			\rho(\phi) = \lim_{t\to\infty} \frac1t  \log \# P_t(\phi) \leq
			\lim_{t\to\infty}\frac1t \log R(t,\ep) = h_{top}(\phi).
		\]
	\end{proof}
	
	\begin{prop}
	The geodesic flow of a compact Benoist 3-manifold has positive topological entropy. 
	\label{prop:positiveentropy}
	\end{prop}
	\begin{proof}
		By Corollary \ref{cor:exist_hyp_elts}, there exist noncommuting hyperbolic
		elements $g,h\in\Gam$ which generate a free subgroup. There
		is a positive lower bound for the 
		exponential growth rate of lengths of closed curves associated to this subgroup,
		which bounds below $\rho(\phi)$ and hence $h_{top}(\phi)$. 
	\end{proof}

\subsection{Volume entropy}

	We remark in this section that A. Manning's proof that volume entropy and topological
	entropy agree for compact nonpositively curved Riemannian manifolds   generalizes to our
	context 
	immediately \cite{manning}. 
	Let $\vol$ be a uniform volume on $\Om$, meaning the volumes of unit metric balls are
	uniformly
	bounded above and below by positive constants. There is no canonical such volume but there
	are good candidates (c.f. beginning of \cite[pp 207-261]{handbook}). 
	Then 
	\[
		h_{vol}(\Om) = \displaystyle\lim_{r\to\infty} \frac1r \log \vol(B_\Om(x,r))
	\]
	is the volume entropy of $\Om$. 
	Let $\delta_\Gam$ denote the critical exponent of the action of $\Gam$ on $\Om$, 
	equivalently: $\delta_\Gam = \limsup_{t\to\infty}\frac1t\log N_\Gam(t)$ where
	$N_\Gam(t)=\#\{\gam\in \Gam\mid d_\Om(x,\gam.x)\leq t\}$.

	\begin{prop} 
		If $\Om$ is any divisible properly convex domain in $\RP^n$,
		then 
		\[
			\delta_\Gam=h_{vol}=h_{top}(\phi).
		\]
		\label{prop:vol_entropy}
	\end{prop}
	\begin{proof}
		Whenever a discrete group of isometries acts cocompactly on a metric space with
		finite critical exponent, 
		one has $\delta_\Gam=h_{vol}$ (a proof is available in \cite[Lemma 4.5]{quintnotes}).
		The statement in  \cite[Theorem 1]{manning} that $h_{vol}\leq h_{top}(\phi)$
		holds as long as $M$ is compact and $(\Om,d_\Om)$ is complete. The proof of the
		opposite inequality in Theorem 2 uses nonpositive sectional curvature to prove a
		technical lemma. We can bypass curvature and prove the lemma immediately in Hilbert
		geometries. The rest of the proof follows in the same way. 
	
		This lemma has already been proven by Crampon in the strictly convex case, but in 
		the proof
		Crampon only uses strict convexity to state the lemma with a strict inequality 
		for all {\em geodesics}, which can only be projective lines when the domain is strictly
		convex.
		Since our geodesic flow is defined to follow projective lines, the lemma suffices. 
	
	\begin{lemma}[\cite{Cr09}, Lemma 8.3]
		For any two projective lines $\sigma,\tau\colon [0,r]\to M$, $r>0$, 
		\[ d_\Om(\sigma(t),\tau(t)) \leq d_\Om(\sig(0),\tau(0))+d_\Om(\sig(r),\tau(r)).
		\]
	\end{lemma}
	\end{proof}
	
%

	\bibliographystyle{plain}

	\bibliography{biblio_submitted.bib}

\end{document}